\documentclass[12pt,reqno]{amsart}
\usepackage{amssymb}
\usepackage{amsfonts}
\usepackage{color}
\usepackage{hyperref}
\usepackage{amsaddr}
\usepackage{tkz-graph}
\usetikzlibrary{arrows.meta,arrows,positioning,quotes}
\usepackage{siunitx}

\newtheorem{theorem}{Theorem}[section]

\def\D{\mathcal{D}}
\def\B{\mathcal{B}}
\def\E{\mathcal{E}}
\DeclareMathOperator{\der}{der}
\DeclareMathOperator{\Aut}{Aut}
\DeclareMathOperator{\AG}{AG}
\DeclareMathOperator{\PG}{PG}
\DeclareMathOperator{\Mob}{\textnormal{M\"{o}b}}
\DeclareMathOperator{\AGL}{AGL}
\DeclareMathOperator{\PSL}{PSL}
\DeclareMathOperator{\PGL}{PGL}
\DeclareMathOperator{\SQS}{SQS}
\DeclareMathOperator{\CSQS}{CSQS}
\DeclareMathOperator{\RoSQS}{RoSQS}
\DeclareMathOperator{\STS}{STS}
\DeclareMathOperator{\SL}{SL}

\begin{document}

\title[Steiner $3$-designs as extensions]{Steiner $3$-designs as extensions$^\star$}

\author[M.~Kiermaier, V.~Kr\v{c}adinac, and A.~Wassermann]{Michael Kiermaier$^1$, Vedran Kr\v{c}adinac$^2$,\\and Alfred Wassermann$^1$}

\address{$^1$Department of Mathematics, University of Bayreuth, 95440 Bayreuth, Germany}

\address{$^2$Faculty of Science, University of Zagreb, Bijeni\v{c}ka cesta~30, HR-10000 Zagreb, Croatia}


\email{michael.kiermaier@uni-bayreuth.de}
\email{vedran.krcadinac@math.hr}
\email{alfred.wassermann@uni-bayreuth.de}

\thanks{$^\star$Part of this work was done while the second author was
visiting the University of Bayreuth supported by a scholarship from the
German Academic Exchange Service (DAAD)}



\date{September 27, 2025}

\begin{abstract}
In this article, we construct a Steiner system with the parameters $S(3,6,42)$, settling one of the smallest open parameter sets of Steiner $3$-designs. Furthermore, we establish the existence of rotational Steiner quadruple systems on $46$ and $92$ points.

Our construction method is based on extending Steiner $2$-designs using prescribed extension groups. We also consider extensions to designs of higher strength. The article includes a table and a discussion of the status of all admissible parameters for Steiner $3$-designs on at most $50$ points.
\end{abstract}

\maketitle

\section{Introduction}

Steiner systems are among the most fascinating objects in
Design Theory, if not all of Combinatorics. They have
a long history~\cite{DR80} and many applications, as well as
connections to other important combinatorial structures~\cite{CM07}.
A \emph{Steiner system} $\D=(X,\B)$ with parameters $S(t,k,v)$
comprises a set~$X$ of $v$ \emph{points}, and a
family~$\B$ of $k$-element subsets of~$X$ called \emph{blocks}
such that every $t$-subset of points is contained in exactly
one block. Alternatively, $\D$ is called a \emph{Steiner
$t$-design} and the parameters are written as $t$-$(v,k,1)$.
A necessary condition for the existence of Steiner systems
is that
\begin{equation}\label{divcond}
{k-i\choose t-i} \mbox{ divides } {v-i\choose t-i},\kern 2mm
\mbox{ for all } i\in\{0,\ldots,t\}.
\end{equation}
Given a point $P\in X$, the \emph{derived design} $\der_P \D$
has $X\setminus\{P\}$ as points and
$\{B\setminus\{P\} \mid B\in \B,\, P\in B\}$ as blocks,
and is a Steiner system $S(t-1,k-1,v-1)$ (for $t\ge 1$).
We refer to~\cite{BJL99} and~\cite{CD07} for other results
and background material about combinatorial designs.

It was proved by probabilistic arguments that Steiner
$t$-designs exist for all~$t$~\cite{PK24, GKLO23}.
However, explicit examples are known only for $t\le 5$,
and only finitely many explicit examples are known for $t\ge 4$.
Many known designs were constructed by prescribing
groups of automorphisms and using computational methods,
such as the Kramer-Mesner method~\cite{KM76}.
In the recent paper~\cite{KKTVKW25}, this method was
fine-tuned for Steiner systems. Implementations of the
different computational steps are available in the
GAP~\cite{GAP} package \emph{Prescribed Automorphism
Groups}~\cite{PAG}. The paper~\cite{KKTVKW25} focuses
on Steiner $2$-designs, but the method can
be used for arbitrary $S(t,k,v)$ designs, subject
to time and memory limitations.

The purpose of this paper is to construct explicit examples
of small Steiner $3$-designs. There are $25$ admissible
parameter sets $S(3,k,v)$ with $v\le 50$. Previously,
designs were known for $20$ of these parameter sets,
nonexistence was established for one, and existence was an
open problem for the remaining four. Our main result is the
construction of a design with parameters $S(3,6,42)$, settling
one of the open cases. Furthermore, we establish the existence
of rotational Steiner quadruple systems for $v=46$
and $v=92$, and we construct many new Steiner designs
with parameters for which existence was already known.
We made all constructed designs available in GAP-readable
format, compatible with the DESIGN~\cite{LHS24} package.
They can be accessed as a Zenodo data set~\cite{Zenodo}
and through our web page
\begin{center}
\url{https://web.math.hr/~krcko/results/steiner3.html}.
\end{center}

In Section~\ref{sec2}, we present a table of admissible
parameters $S(3,k,v)$ with $v\le 50$, containing information
on the number of non-isomorphic designs and extensibility.
Examples of the newly established designs are given by
listing the prescribed groups and base blocks
(Theorems~\ref{rosqs46}, \ref{rosqs92}, and~\ref{tm1}).

In Section~\ref{sec3}, we describe a procedure for
extending Steiner systems with prescribed extension
groups, based on the Kramer-Mesner method. The new
$S(3,6,42)$ design was constructed in this way. We
also tried constructing designs for the remaining three
open parameter sets $S(3,6,46)$, $S(3,5,41)$, and
$S(3,5,50)$. No such design could be found, but many
new Steiner $2$-designs for the derived parameters
$S(2,5,45)$, $S(2,4,40)$, and $S(2,4,49)$ were
constructed in this process (Theorems~\ref{tm2-5-45},
\ref{tm2-4-40}, and~\ref{tm2-4-49}).

Finally, in Section~\ref{sec4}, we discuss extensions
of $S(3,k,v)$ designs to $t$-designs of strength
$t>3$. Using the procedure from the previous
section, we prove that the new $S(3,6,42)$ design
does not extend to a $4$-design if any nontrivial
extension group is assumed (Theorem~\ref{tm4-7-43}).
We also prove that the only known $S(5,7,28)$ design
does not extend to a $6$-design
(Theorem~\ref{tm6-8-29}).

\section{A table of admissible parameters}\label{sec2}

Table~\ref{tab1} contains all admissible parameters $S(3,k,v)$ with
$v\le 50$. Admissibility means that the divisibility
conditions~\eqref{divcond} and
Fisher's inequality for the derived designs are satisfied. The latter
condition eliminates only the parameters $S(3,7,22)$ in our range.
The first two columns contain~$v$ and~$k$, followed by the number~$b$
of blocks. The fourth column labeled `Nd' contains the number
of $S(3,k,v)$ designs up to isomorphism, or a lower bound if the
exact number is not known. Let~$t_{\max}$ be the largest integer~$t$
such that a $S(t,k+t-3,v+t-3)$ design exists. Note that, by taking
derived designs, all $S(t,k+t-3,v+t-3)$ with $t \leq t_{\max}$ do exist.
The fifth and the sixth columns contain the sharpest known lower and
upper bounds on~$t_{\max}$, respectively. The last two columns
contain references and comments.

\begin{table}[t]
\begin{tabular}{ccccccll}
\hline
$v$ & $k$ & $b$ & Nd & $t_{\max}$$\ge$ & $t_{\max}$$\le$ & References & Comments\\
\hline
8 & 4 & 14 & 1 & 3 & 3 & \cite{JAB08, EW38} & $\AG_2(3,2)$\\	
10 & 4 & 30 & 1 & 5 & 5 & \cite{JAB08, EW38, JAT94} & $\Mob(3)$ \\
14 & 4 & 91 & 4 & 3 & 3 & \cite{MH72} & \\	
16 & 4 & 140 & \num{1054163} & 3 & 3 & \cite{KOP06, OP08} & $\AG_2(4,2)$\\
17 & 5 & 68 & 1 & 3 & 3 & \cite{EW38} & $\Mob(4)$ \\
20 & 4 & 285 & $\ge$\num{60077} & 3 & 15 & & \\
22 & 4 & 385 & $\ge$\num{52240} & 5 & 17 & \cite{RHFD76} & \\
22 & 6 & 77 & 1	& 5 & 5 & \cite{EW38} &  \\
26 & 4 & 650 & $\ge$\num{51145} & 3 & 21 & & \\
26 & 5 & 260 & $\ge$1 & 5 & 8 & \cite{RHFD76,GGP87} &  \\
26 & 6 & 130 & 1 & 3 & 3 & \cite{YC72, RHFD73a, JAT94} & $\Mob(5)$ \\
28 & 4 & 819 & $\ge$\num{12149} & 3 & 3 & & Spherical \\
32 & 4 & \num{1240} & $\ge$\num{1516} & 3 & 27 & & $\AG_2(5,2)$\\	
34 & 4 & \num{1496} & $\ge$\num{1569} & 5 & 29 & \cite{BLW99} & \\	
37 & 7 & 222 & 0 & 1 & 1 & & $\Mob(6)$ \\
38 & 4 & \num{2109} & $\ge$\num{1547} & 3 & 3 & & \\	
40 & 4 & \num{2470} & $\ge$\num{1557} & 3 & 35 & & \\	
41 & 5 & \num{1066} & ? & 2 & 16 & & \\	
42 & 6 & 574 & $\ge$1 & 3 & 9 & & Theorem~\ref{tm1} \\	
44 & 4 & \num{3311} & $\ge$\num{1504} & 3 & 39 & & \\	
46 & 4 & \num{3795} & $\ge$\num{1559} & 5 & 41 & \cite{GGM92} & \\	
46 & 6 & 759 & ? & 2 & 3& & \\
50 & 4 & \num{4900} & $\ge$\num{1535} & 3 & 45 & & \\	
50 & 5 & \num{1960} & ? & 2 & 20 & & \\	
50 & 8 & 350 & 1 & 3 & 3 & \cite{RHFD73b, JAT94} & $\Mob(7)$ \\
\hline
\end{tabular}
\vskip 4mm
\caption{Steiner $3$-designs with at most $50$ points.}\label{tab1}
\end{table}

We first discuss the parameters from Table~\ref{tab1} belonging to known
infinite families of Steiner $3$-designs. The most widely studied family
are Steiner quadruple systems $S(3,4,v)$, abbreviated as $\SQS(v)$. Two
comprehensive surveys are~\cite{LR78} and~\cite{HP92}. $\SQS(v)$ are known to
exist for all admissible~$v$, i.e.\ for $v\equiv 2,4 \pmod{6}$, $v\ge 8$.
This was first proved in~\cite{HH60}, and simplifications
were given in~\cite{AH82, AH94, HL85, ZG13}.
A subfamily with parameters $S(2,4,2^n)$, $n\ge 3$ are the $n$-dimensional
affine spaces $\AG_2(n,2)$ over the binary field, with planes as blocks.
This family was already described by T.~P.~Kirkman~\cite{TPK847}.
Another subfamily are the spherical geometries of order $q=3$
with parameters $S(3,4,3^n+1)$, $n\ge 2$, discussed below.

The uniqueness of $\SQS(8)$ and $\SQS(10)$ was proved in~\cite{JAB08}
and~\cite{EW38}. In~\cite{MH72}, it was computationally proved that
there are exactly four $\SQS(14)$ up to isomorphism. A computer
enumeration of $\SQS(16)$ was performed in~\cite{KOP06}, requiring
about $12$ years of CPU time. We repeated the calculation using
different software and confirmed $\num{1054163}$ as the number of
$\SQS(16)$ up to isomorphism; see Section~\ref{sec3} for more
details. In~\cite{LR76}, it was established that there are at
least~$10^{17}$ non-isomorphic $\SQS(20)$, making a computer
enumeration infeasible. Asymptotic estimates for the number of
non-isomorphic $\SQS(v)$ for sufficiently large admissible~$v$ are
known~\cite{DV71, HL85a}.

In Table~\ref{tab1}, we give rough lower bounds on the number of
$\SQS(v)$ for $20\le v\le 50$, based on explicitly constructed examples.
It is clear that the actual numbers `Nd' are many orders of magnitude
larger. Our lower bounds are primarily due to memory considerations,
since we are making the designs available in~\cite{Zenodo}.
For $v\ge 32$, we stopped the search when the number of constructed
designs exceeded $1500$. The designs were constructed by prescribing
groups of automorphisms and using the implementation of the
Kramer-Mesner method from~\cite{PAG}, described in~\cite{KKTVKW25}.
\emph{Nauty / Traces}~\cite{MP14} was used for isomorphism tests
and to compute the full automorphism groups.

We aimed to choose the prescribed groups large enough to
make the computation feasible, while still producing many
non-isomorphic designs. For some parameters, there is a
``classical'' design~$\D$ with a large full automorphism group.
For example, one $\SQS(28)$ is the spherical geometry for
$q=n=3$, with full automorphism group $\Aut(\D)\cong
\PSL(2,27)\rtimes C_6$ of order \num{58968}. One example of a
$\SQS(32)$ is $\AG_2(5,2)$, with $\Aut(\D)\cong \AGL(5,2)$ of
order \num{319979520}. We can take subgroups $G\le \Aut(\D)$
and construct designs with~$G$ as prescribed group, possibly
getting other designs besides~$\D$. When there is no such
``classical'' design, we prescribed extensions of regular
permutation representations of groups of orders~$v$ and~$v-1$,
utilizing the GAP~\cite{GAP} library of small groups.

An $\SQS(v)$ is called \emph{cyclic} and denoted by $\CSQS(v)$
if it has a cyclic automorphism of order~$v$. It is called
\emph{rotational} and denoted by $\RoSQS(v)$ if it
has a cyclic automorphism of order~$v-1$ with a fixed point~$\infty$.
Designs of type $\CSQS(v)$ and $\RoSQS(v)$ have been extensively studied,
but the sets of all values~$v$ for which they exist have not
been fully determined. For $\CSQS(v)$, there is only one
unknown value below $100$, namely $v=94$~\cite[Theorem~4.6]{FC11}.
We were not able to construct a $\CSQS(94)$ with our method,
nor to rule it out. For $\RoSQS(v)$, there were seven open
values below $100$: $v=46$, $56$, $70$, $82$, $86$, $92$,
and $98$~\cite[Table~I]{JZ02}. We were able to construct a
$\RoSQS(46)$ and a $\RoSQS(92)$. Examples can be downloaded
from~\cite{Zenodo}, and here we describe one of each
by giving base blocks.

\allowdisplaybreaks

\begin{theorem}\label{rosqs46}
There exists a $\RoSQS(46)$.
\end{theorem}

\begin{proof}
Let $G$ be the group $\mathbb{Z}_{45}=\{0,\ldots,44\}$ with addition
modulo~$45$. Then the union of the $G$-orbits of the $4$-subsets listed
below forms a $\RoSQS(46)$:
\begin{align*}
 & \{ 0, 1, 2, 21 \} & & \{ 0, 1, 3, 24 \} & & \{ 0, 1, 4, 22 \} & & \{ 0, 1, 5, 42 \} & & \{ 0, 1, 6, 41 \} \\[1mm]
 & \{ 0, 1, 7, 31 \} & & \{ 0, 1, 8, 27 \} & & \{ 0, 1, 9, 19 \} & & \{ 0, 1, 10, 13 \} & & \{ 0, 1, 11, \infty \} \\[1mm]
 & \{ 0, 1, 12, 37 \} & & \{ 0, 1, 14, 35 \} & & \{ 0, 1, 15, 25 \} & & \{ 0, 1, 16, 40 \} & & \{ 0, 1, 17, 26 \} \\[1mm]
 & \{ 0, 1, 18, 28 \} & & \{ 0, 1, 23, 33 \} & & \{ 0, 1, 29, 38 \} & & \{ 0, 1, 30, 34 \} & & \{ 0, 1, 32, 36 \} \\[1mm]
 & \{ 0, 1, 39, 43 \} & & \{ 0, 2, 4, 35 \} & & \{ 0, 2, 5, 29 \} & & \{ 0, 2, 6, 26 \} & & \{ 0, 2, 7, 28 \} \\[1mm]
 & \{ 0, 2, 8, 14 \} & & \{ 0, 2, 9, 36 \} & & \{ 0, 2, 10, 17 \} & & \{ 0, 2, 11, 40 \} & & \{ 0, 2, 12, 22 \} \\[1mm]
 & \{ 0, 2, 13, 31 \} & & \{ 0, 2, 15, 19 \} & & \{ 0, 2, 16, \infty \} & & \{ 0, 2, 18, 38 \} & & \{ 0, 2, 20, 24 \} \\[1mm]
 & \{ 0, 2, 25, 39 \} & & \{ 0, 2, 27, 34 \} & & \{ 0, 2, 30, 42 \} & & \{ 0, 2, 32, 37 \} & & \{ 0, 3, 6, 13 \} \\[1mm]
 & \{ 0, 3, 7, \infty \} & & \{ 0, 3, 9, 30 \} & & \{ 0, 3, 11, 32 \} & & \{ 0, 3, 12, 16 \} & & \{ 0, 3, 14, 38 \} \\[1mm]
 & \{ 0, 3, 15, 37 \} & & \{ 0, 3, 17, 20 \} & & \{ 0, 3, 18, 39 \} & & \{ 0, 3, 19, 22 \} & & \{ 0, 3, 23, 34 \} \\[1mm]
 & \{ 0, 3, 25, 40 \} & & \{ 0, 4, 9, 26 \} & & \{ 0, 4, 11, 28 \} & & \{ 0, 4, 12, 21 \} & & \{ 0, 4, 17, 31 \} \\[1mm]
 & \{ 0, 4, 18, 37 \} & & \{ 0, 4, 19, 39 \} & & \{ 0, 4, 20, 34 \} & & \{ 0, 4, 23, 38 \} & & \{ 0, 4, 29, 40 \} \\[1mm]
 & \{ 0, 5, 11, 19 \} & & \{ 0, 5, 12, 28 \} & & \{ 0, 5, 14, 36 \} & & \{ 0, 5, 17, 32 \} & & \{ 0, 5, 18, 25 \} \\[1mm]
 & \{ 0, 5, 20, 37 \} & & \{ 0, 5, 23, \infty \} & & \{ 0, 5, 24, 35 \} & & \{ 0, 5, 27, 39 \} & & \{ 0, 5, 31, 38 \} \\[1mm]
 & \{ 0, 6, 15, 38 \} & & \{ 0, 6, 16, 22 \} & & \{ 0, 6, 17, 23 \} & & \{ 0, 6, 18, 36 \} & & \{ 0, 6, 19, \infty \} \\[1mm]
 & \{ 0, 6, 20, 32 \} & & \{ 0, 7, 15, 29 \} & & \{ 0, 7, 17, 33 \} & & \{ 0, 8, 16, 32 \} & & \{ 0, 8, 19, 35 \} \\[1mm]
 & \{ 0, 8, 21, 33 \} & & \{ 0, 8, 25, \infty \} & & \{ 0, 9, 21, \infty \} & & \{ 0, 9, 22, 35 \} & & \{ 0, 15, 30, \infty \}
\end{align*}
\end{proof}

\begin{theorem}\label{rosqs92}
There exists a $\RoSQS(92)$.
\end{theorem}

\begin{proof}
Let $G$ be the subgroup of the symmetric group on
$\mathbb{Z}_{91}$ generated by $x \mapsto x+1$
and $x \mapsto 4x$. Then, $G$ is of order $546$ and isomorphic
to a semidirect product $C_{91}\rtimes C_6$. The union of the
$G$-orbits of the $4$-subsets listed below forms a $\RoSQS(92)$:
\begin{align*}
 & \{ 0, 1, 2, 13 \} & &  \{ 0, 1, 3, 52 \} & &  \{ 0, 1, 4, 29 \} & &  \{ 0, 1, 5, 31 \} & &  \{ 0, 1, 6, 17 \} \\[1mm]
 & \{ 0, 1, 7, 59 \} & &  \{ 0, 1, 8, 26 \} & &  \{ 0, 1, 9, 48 \} & & \{ 0, 1, 10, 66 \} & &  \{ 0, 1, 11, \infty  \} \\[1mm]
 & \{ 0, 1, 14, 39 \} & &  \{ 0, 1, 15, 27 \} & &  \{ 0, 1, 16, 43 \} & &  \{ 0, 1, 18, 76 \} & &  \{ 0, 1, 19, 58 \} \\[1mm]
 & \{ 0, 1, 20, 85 \} & &  \{ 0, 1, 21, 40 \} & &  \{ 0, 1, 22, 61 \} & & \{ 0, 1, 24, 36 \} & &  \{ 0, 1, 25, 60 \} \\[1mm]
 & \{ 0, 1, 28, 79 \} & &  \{ 0, 1, 32, 84 \} & &  \{ 0, 1, 33, 63 \} & &  \{ 0, 1, 34, 45 \} & &  \{ 0, 1, 35, 83 \} \\[1mm]
 & \{ 0, 1, 37, 55 \} & & \{ 0, 1, 38, 71 \} & &  \{ 0, 1, 41, 47 \} & &  \{ 0, 1, 44, 70 \} & &  \{ 0, 1, 46, 67 \} \\[1mm]
 & \{ 0, 1, 50, 77 \} & &  \{ 0, 1, 54, 80 \} & &  \{ 0, 1, 56, 89 \} & &  \{ 0, 1, 57, 62 \} & &  \{ 0, 1, 64, 78 \} \\[1mm]
 & \{ 0, 1, 72, 82 \} & & \{ 0, 1, 73, 81 \} & &  \{ 0, 1, 75, 86 \} & &  \{ 0, 2, 5, 71 \} & &  \{ 0, 2, 7, 44 \} \\[1mm]
 & \{ 0, 2, 8, 30 \} & &  \{ 0, 2, 9, 33 \} & &  \{ 0, 2, 10, 67 \} & &  \{ 0, 2, 11, 14 \} & & \{ 0, 2, 12, 34 \} \\[1mm]
 & \{ 0, 2, 20, 42 \} & &  \{ 0, 2, 24, 48 \} & &  \{ 0, 2, 32, 57 \} & &  \{ 0, 2, 35, 72 \} & &  \{ 0, 2, 36, 61 \} \\[1mm]
 & \{ 0, 2, 37, \infty  \} & &  \{ 0, 2, 41, 56 \} & &  \{ 0, 2, 43, 63 \} & &  \{ 0, 2, 55, 60 \} & & \{ 0, 2, 59, 81 \} \\[1mm]
 & \{ 0, 3, 9, 21 \} & &  \{ 0, 3, 18, 51 \} & &  \{ 0, 3, 31, 53 \} & &  \{ 0, 3, 36, 72 \} & &  \{ 0, 3, 43, 76 \} \\[1mm]
 & \{ 0, 3, 50, 56 \} & &  \{ 0, 3, 55, 84 \} & & \{ 0, 5, 11, 66 \} & &  \{ 0, 5, 14, 61 \} & &  \{ 0, 5, 30, 85 \} \\[1mm]
 & \{ 0, 5, 33, 77 \} & &  \{ 0, 5, 41, 82 \} & &  \{ 0, 5, 47, 65 \} & &  \{ 0, 5, 49, 70 \} & &  \{ 0, 7, 21, 63 \} \\[1mm]
 & \{ 0, 7, 28, 42 \} & &  \{ 0, 7, 36, \infty  \} & & \{ 0, 13, 26, 52 \} & &  \{ 0, 13, 65, \infty  \} & &  \{ 0, 15, 73, \infty \}
\end{align*}
\end{proof}

Steiner systems $S(3,q+1,q^n+1)$ are known as \emph{spherical
geometries} or \emph{inversive spaces}. They exist for all
$n\ge 2$ and prime powers~$q$~\cite{EW38}. The most studied
case is $n=2$, called \emph{inversive} or \emph{M\"{o}bius
planes} and denoted by $\Mob(q)$ in Table~\ref{tab1}. Classical
inversive planes are called \emph{Miquelian} because they are
characterized by Miquel's theorem~\cite[Theorem~6.1.5]{PD97}.
The only known non-Miquelian examples have orders $q=2^{2e+1}$,
$e\ge 1$, and are associated with the ovoids of Tits~\cite{JT62}.

The uniqueness of $\Mob(q)$ was proved for $q=4$ in~\cite{EW38},
for $q=5$ in~\cite{YC72} and~\cite{RHFD73a}, and for
$q=7$ in~\cite{RHFD73b}. For $q=3,5,7$, uniqueness also
follows from a theorem of Thas~\cite{JAT94} and the
uniqueness of affine planes $\AG(2,q)$. For similar results
about larger orders~$q$, see~\cite{TP22} and the references
therein. Designs $\Mob(6)$ do not exist because their derived
designs would be affine planes of order~$6$.

There are three ``sporadic'' designs in Table~\ref{tab1},
not belonging to the infinite families discussed above. The
first is the twice derived Witt $5$-design with parameters
$S(3,6,22)$, also known to be unique by~\cite{EW38}. The
second has parameters $S(3,5,26)$ and was investigated
in~\cite{GGP87}, where the question was raised whether
this design is in fact unique. In Section~\ref{sec3},
we explore the possibility of a computational approach
to this question, which remains open. The third sporadic
example is the following new $S(3,6,42)$ design.

\begin{theorem}\label{tm1}
There exists an $S(3,6,42)$ design.
\end{theorem}

\begin{proof}
Let $G=\langle \alpha,\beta\rangle\le S_{42}$ be the permutation group
generated by
\begin{align*}
\alpha =\, & (1,2,4)(3,9,7)(5,6,8)(10,11,13)(12,18,16)(14,15,17)(19,38,34) \\
 & (20,39,36)(21,37,35)(22,42,32)(23,40,31)(24,41,33)(28,30,29),\\[2mm]
\beta =\, & (1,10)(2,11)(3,12)(4,16)(5,17)(6,18)(7,13)(8,14)(9,15)(19,23)\\
 & (20,24)(21,22)(25,35)(26,36)(27,34)(28,33)(29,31)(30,32)\\
 & (38,39)(41,42).
 \end{align*}
This is a group of order $432$ isomorphic to $\AGL(2,3)$.
Then the union of the $G$-orbits of the $6$-subsets listed
below is a $S(3,6,42)$:
\begin{align*}
 & \{ 1, 2, 3, 10, 11, 12 \} & & \{ 1, 2, 4, 20, 36, 39 \} & & \{ 1, 2, 13, 17, 31, 34 \}\\[1mm]
 & \{ 1, 10, 19, 22, 25, 30 \} & & \{ 1, 11, 20, 21, 29, 34 \} & & \{ 19, 20, 21, 22, 23, 24 \}\\[1mm]
 & \{ 19, 22, 31, 36, 38, 42 \}
\end{align*}
\end{proof}

The group~$G$ from the proof of Theorem~\ref{tm1} is the full automorphism
group of the design~$\D$. It partitions the point set into two orbits, one of
size~$18$ and the other of size~$24$. The unusual action of $\Aut(\D)$ might
be the reason why~$\D$ remained undiscovered for so long. Another Steiner
design constructed from a non-transitive group is the $S(5,6,36)$
from~\cite{BLW99}, where a permutation group comprising two copies
of the natural action of $\PGL(2,17)$ on~$18$ points was used.
Our new design has $\Aut(\D)\cong \AGL(2,3)$, suggesting a connection
with the affine plane of order~$3$.

Three parameter sets in Table~\ref{tab1} remain open with regards to
existence: $S(3,5,41)$, $S(3,6,46)$, and $S(3,5,50)$. In the next section,
we describe the computational approach that led to the discovery of the
$S(3,5,42)$ design and our attempts to construct designs for the
remaining open parameter sets.

\section{Extending designs with prescribed groups}\label{sec3}

Let $\E$ be a $S(t,k,v)$ design with a group~$G$ of
automorphisms fixing a point~$\infty$. Then the
derived design $\D=\der_\infty \E$ is also $G$-invariant.
In this situation, we refer to~$\E$ as an \emph{extension}
of~$\D$ and to~$G$ as an \emph{extension group}. Clearly,
$G$ must be a subgroup of both $\Aut\D$ and $\Aut\E$.
For example, a rotational $\SQS(v)$ is the extension
of a cyclic Steiner triple system $\STS(v-1)$ with
$C_{v-1}$ 
as extension group. We use the customary
abbreviation $\STS(v)$ for $S(2,3,v)$ designs, known
as Steiner triple systems.

Our strategy for constructing Steiner $3$-designs is
to extend Steiner $2$-designs with prescribed extension
groups. Typically, we first classify $S(2,k-1,v-1)$
designs~$\D$ with a fixed prescribed group~$G$. Depending on
the parameters, this is a significantly smaller computation
than classifying the corresponding $3$-designs directly.
Then, for each such design $\D$, we add $\infty$ to all
blocks of~$\D$ and then compute all extensions to a $G$-invariant
$S(3,k,v)$ design~$\E$. The $G$-invariance of~$\E$ precisely
means that the extension consists of complete $G$-orbits of
$k$-subsets that do not contain~$\infty$.
For a fixed~$\D$, this also requires less computation than
a full classification of $G$-invariant designs~$\E$, because
the blocks through~$\infty$ are already given. The $k$-subset
orbits covering any $3$-subset already covered by~$\D$ can be
discarded, reducing the size of the Kramer-Mesner system. The
extension process is a modification of the standard
Kramer-Mesner algorithm, available in the package~\cite{PAG}
as a command \texttt{ExtendSteinerDesign}.

In~\cite{MH72} and~\cite{KOP06}, $\SQS(14)$ and $\SQS(16)$
were classified by extending $\STS(13)$ and $\STS(15)$,
respectively. To get a full classification, all non-isomorphic
$\STS$ must be extended with the trivial extension group.
The $80$ non-isomorphic $\STS(15)$ were determined
in~\cite{HS55}. In~\cite{KOP06}, extending each $\STS(15)$
was set up as an instance of the exact cover problem,
and D.~E.~Knuth's \emph{dancing links} algorithm~\cite{DEK00}
was used to solve thes problems.
We repeated the calculation using the \emph{ssxcc}~\cite{DEK23}
implementation of the more recent \emph{dancing cells}
algorithm~\cite{DEK25}. The obtained systems $\SQS(16)$ need
to be checked for isomorphism, but this requires less CPU time
than the extension phase. We used the \texttt{BlockDesignFilter}
command from~\cite{PAG}, relying on \emph{nauty / Traces}~\cite{MP14}.
We got the same total number of non-isomorphic $\SQS(16)$
and distribution by full automorphism group orders
as in~\cite[Table~1]{KOP06}.

One could attempt to prove the uniqueness of the $S(3,5,26)$
design in this manner. The derived designs have parameters
$S(2,4,25)$ and were classified in~\cite{ES96}; there
are $18$ examples up to isomorphism. However, in this case
the extension phase seems infeasible, because each $S(2,4,25)$
gives rise to a very large exact cover problem.
In Knuth's terminology~\cite{DEK00}, each problem has $2100$
items and $\num{16380}$ options, out of which $210$ are to be chosen.
We couldn't solve these exact cover problems with the resources
available to us. For comparison, the $80$ exact cover
problems in the classification of $\SQS(16)$ have $420$
items and $945$ options, out of which $105$ are to be
chosen.

\begin{table}[!b]
\begin{tabular}{cc}
\hline
$\lvert\Aut(\D)\rvert$ & Nd \\
\hline
205	& 1\\
120	& 2\\
24 & 2\\
20 & 1\\
18 & 4\\
12 & 1\\
9 & 1\\
6 & 2\\
2 & 1\\
\hline
\end{tabular}
\vskip 4mm
\caption{Distribution of the known $S(2,5,41)$ designs by full automorphism group orders.}\label{tab2}
\end{table}

When a complete classification is out of reach, prescribing
a nontrivial extension group~$G$ may reduce the computation
to manageable size. An added benefit is that we have a
natural way of choosing~$G$: It must be a subgroup of~$\Aut(\D)$.
For example, $15$ non-isomorphic $S(2,5,41)$ designs are known,
with full automorphism group orders given in Table~\ref{tab2}.
Among them are all designs with automorphisms of order~$3$,
classified in~\cite{VK02}. The known $S(2,5,41)$ designs are
included in the Zenodo data set~\cite{Zenodo}. Two of
the designs, with $\lvert\Aut(\D)\rvert=24$ and $\lvert\Aut(\D)\rvert=18$, extend
to a $S(3,6,42)$ design~$\E$. For both designs~$\D$, the full
automorphism group can be used as extension group. The group
of order $24$ is large enough to construct the $S(3,6,42)$
design~$\E$ even without prescribing the derived subdesign,
but it would be difficult to guess its action on~$42$ points
without the knowledge of the derived design.
The group is isomorphic to $\SL(2,3)$ and acts in five orbits
of sizes $1$, $1$, $8$, $8$, and $24$. Instead of guessing,
we get the extension groups from the prescribed derived
subdesigns~$\D$. We examined all known $S(2,5,41)$ designs
with $G\le \Aut(\D)$, $\lvert G\vert\ge 10$ as extension groups, and
found only one $S(3,6,42)$ design up to isomorphism, given
in Theorem~\ref{tm1}.

We tried to construct designs for the remaining three open
parameter sets from Table~\ref{tab1} by following the same
procedure, but no designs were found. Derived designs of
a hypothetical $S(3,6,46)$ have parameters $S(2,5,45)$.
Previously, $65$ non-isomorphic $S(2,5,45)$ designs were
known~\cite{CZ22}; see also~\cite{CDDRS19} and references
therein. We managed to construct over $1000$ new examples
by prescribing various groups of orders~$16$ and~$8$, and
by applying so-called paramodifications~\cite{MN21}.

\begin{theorem}\label{tm2-5-45}
There exist at least $1072$ non-isomorphic $S(2,5,45)$ designs.
Their distribution by orders of full automorphism groups is
given in Table~\ref{tab3}.
\end{theorem}

\begin{table}[t]
\begin{tabular}{cc}
\hline
$\lvert\Aut(\D)\rvert$ & Nd \\
\hline
360 & 1 \\
160 & 1 \\
72 & 6 \\
40 & 1 \\
32 & 3 \\
24 & 3 \\
16 & 16 \\
8 & 142 \\
6 & 2 \\
4 & 221 \\
2 & 648 \\
1 & 28 \\
\hline
\end{tabular}
\vskip 4mm
\caption{Distribution of the known $S(2,5,45)$ designs by full automorphism group orders.}\label{tab3}
\end{table}

\begin{table}[!b]
\begin{tabular}{cc@{\rule{10mm}{0mm}}cc@{\rule{10mm}{0mm}}cc}
\hline
$\lvert\Aut(\D)\rvert$ & Nd & $\lvert\Aut(\D)\rvert$ & Nd & $\lvert\Aut(\D)\rvert$ & Nd\\
\hline
\num{12130560} & 1 & 486 & 13 & 96 & 22 \\
\num{151632} & 1 & 480 & 1 & 81 & 43 \\
\num{15552} & 1 & 468 & 1 & 80 & 3 \\
\num{10368} & 1 & 432 & 13 & 78 & 13 \\
\num{7776} & 1 & 384 & 1 & 72 & 467 \\
\num{5184} & 4 & 324 & 153 & 64 & 1 \\
\num{3888} & 5 & 320 & 1 & 60 & 1 \\
\num{2916} & 1 & 288 & 12 & 54 & 11 \\
\num{2592} & 5 & 256 & 3 & 48 & 163 \\
\num{2106} & 3 & 216 & 21 & 40 & 25 \\
\num{1944} & 2 & 192 & 3 & 39 & 194 \\
\num{1458} & 1 & 162 & 19 & 36 & \num{5346} \\
\num{1296} & 18 & 160 & 3 & 32 & 29 \\
972 & 11 & 156 & 3 & 27 & 6 \\
864 & 2 & 144 & 82 & 26 & 176 \\
648 & 57 & 120 & 1 & 24 & \num{2782} \\
576 & 2 & 108 & 169 & 20 & 894 \\
\hline
\end{tabular}
\vskip 4mm
\caption{Distribution of the known $S(2,4,40)$ designs by full automorphism group orders.}\label{tab4}
\end{table}

These designs, as well as the ones from Theorems~\ref{tm2-4-40} and~\ref{tm2-4-49},
are also available in~\cite{Zenodo}. We systematically ran the extension algorithm
with the known examples and extension groups of orders $\lvert G\rvert>10$,
and even with most groups of orders $8$ and $10$ barring
a few hard cases, but we didn't get any $S(3,6,46)$ designs.
The search would take too long for smaller extension groups.

Derived designs of $S(3,5,41)$ have parameters $S(2,4,40)$.
The classical example is $\PG(3,3)$; we managed to construct
many more examples by prescribing various groups of automorphisms
and using the Kramer-Mesner method.

\begin{theorem}\label{tm2-4-40}
There exist at least $\num{10791}$ non-isomorphic $S(2,4,40)$ designs.
Their distribution by orders of full automorphism groups is
given in Table~\ref{tab4}.
\end{theorem}

While trying to extend these designs to $S(3,5,41)$,
we could eli\-mi\-nate all extension groups of orders greater
than~$40$. Some groups of order~$40$ require a lot of
computation and we couldn't examine them exhaustively.
This is also true for groups of orders less than~$40$.
To construct $S(3,5,50)$ designs by the same method,
$S(2,4,49)$ designs are required. Again, we first
applied the usual Kramer-Mesner approach to the
parameters $S(2,4,49)$ with various prescribed
groups.

\begin{theorem}\label{tm2-4-49}
There exist at least $\num{22080}$ non-isomorphic $S(2,4,49)$ designs.
Their distribution by orders of full automorphism groups is
given in Table~\ref{tab5}.
\end{theorem}

\begin{table}[!h]
\begin{tabular}{cc}
\hline
$\lvert\Aut(\D)\rvert$ & Nd \\
\hline
147 & 14 \\
144 & 2 \\
72 & 2 \\
49 & 222 \\
48 & 332 \\
45 & 288 \\
42 & 479 \\
36 & 1672 \\
30 & 1498 \\
27 & 7432 \\
20 & 9879 \\
14 & 260 \\
\hline
\end{tabular}
\vskip 4mm
\caption{Distribution of the known $S(2,4,49)$ designs by full automorphism group orders.}\label{tab5}
\end{table}

The cut-off order for the extension group is now~$72$.
We could eliminate all extensions of the designs from
Theorem~\ref{tm2-4-49} with $\lvert G\rvert\ge 72$, but as one can
see from Table~\ref{tab5}, these groups act on only $18$
non-isomorphic $S(2,4,49)$ designs. There is no classical
example with these parameters, and we did not find many
designs~$\D$ with large enough $\Aut(\D)$ so that the
extension process is computationally feasible.

\section{Extending to higher strengths}\label{sec4}

In Table~\ref{tab1}, the sixth column labeled
`$t_{\max}\le$' contains the largest strength~$t$ for
which a $S(3,k,v)$ design could possibly be extended
to an $S(t,k',v')$, according to the current knowledge of
nonexistence results. Most entries follow
from the divisibility conditions~\eqref{divcond} on $S(t,k',v')$,
with three exceptions. In~\cite{MH72} and in~\cite{OP08},
it was proved by exhaustive computer searches that $S(4,5,15)$
and $S(4,5,17)$ designs do not exist.
Designs $S(4,6,18)$ were ruled out in~\cite[Satz~6]{EW38}
based on a case-by-case analysis. Hence, for $S(3,4,14)$,
$S(3,4,16)$, and $S(3,5,17)$, the bound in the sixth
column is $t_{\max}\le 3$. The fifth column contains
strengths of the largest known extensions. As it happens,
all these (nontrivial) extensions are $5$-designs, with
parameters $S(5,6,12)$, $S(5,6,24)$, $S(5,8,24)$, $S(5,7,28)$,
$S(5,6,36)$, and $S(5,6,48)$. Known designs with these
parameters are also included in~\cite{Zenodo}. Not a single
Steiner $4$-design was discovered directly; all known
examples are actually derived $5$-designs.

We remark that the existence table in \cite[Section~5.3]{CM07}
is organized differently: Admissible parameters $S(t,k,v)$
are grouped according to $k-t$ and $n=v-t$. There, all
entries in the column labeled `$t\le$' follow from
necessary conditions on the parameters, while (non)existence
results are given in separate columns labeled `$\exists$'
and `$\nexists$'. The table was published in 2007, but
only two entries need to be updated with new results. The
first is the nonexistence of $S(4,5,17)$ from~\cite{OP08},
and the second is the new $S(3,6,42)$ design from
Theorem~\ref{tm1}.

In view of the remark about Steiner $4$-designs above,
it would be noteworthy to extend some of the $3$-designs
from Table~\ref{tab1}. The procedure of extending designs
with prescribed groups, used in the previous section
for $t=3$, can be used without change for higher strengths.
In the first place, we searched for extensions of the
new $S(3,6,42)$ design~$\D$ using subgroups $G\le \Aut(\D)$
as extension groups. We eliminated all nontrivial subgroups,
thereby proving the following theorem.

\begin{theorem}\label{tm4-7-43}
If the $S(3,6,42)$ design of Theorem~\ref{tm1}
can be extended to a $S(4,7,43)$ design~$\E$ by
adding a point~$\infty$, then the point stabilizer
$\Aut(\E)_\infty$ is trivial.
\end{theorem}

Eliminating the trivial extension group would require
solving an exact cover problem with $\num{103320}$ items and
$\num{288504}$ options, out of which $2952$ are to be chosen.
We could not perform this computation, but we did manage
to eliminate the trivial extension group for the
following Steiner $5$-design.

\begin{theorem}\label{tm6-8-29}
The Steiner system $S(5,7,28)$ constructed in~\cite{RHFD76}
does not extend to a Steiner $6$-design.
\end{theorem}

Here, an exact cover problem with $\num{343980}$ options and
$\num{39312}$ items, out of which $\num{12285}$ are to be
chosen, was computationally proved to have no solutions.
However, this result does not improve the `$t_{\max}$$\le$'
entry for $S(3,5,26)$ designs, because designs with these
parameters are not known to be unique. Potentially, some
other yet-undiscovered $S(3,5,26)$ design might admit an extension
all the way to a Steiner $8$-design.

The smallest parameters of a Steiner $3$-design for which
the extension problem is open are $S(3,4,20)$. We ran a
lot of extension procedures using extension groups of
orders $\lvert G\rvert\ge 10$, but did not find any
$S(4,5,21)$ design.
Our attempts to find other unknown extensions of the
designs from Table~\ref{tab1} were also unsuccessful.
Unfortunately, these calculations provide very little
evidence that the sought-after Steiner $4$-designs do
not exist. We shall illustrate this with $S(3,4,22)$
designs, for which a few extensions are known.

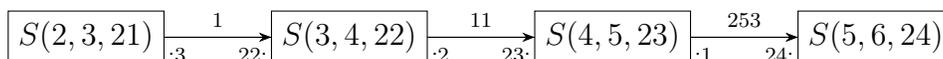
\begin{figure}[!h]
\begin{center}
\begin{tikzpicture}[
    every edge quotes/.style = {auto=right, font=\tiny,
      text=black}]
\GraphInit[vstyle=Empty]
\tikzset{VertexStyle/.style = {shape = rectangle}}
\Vertex[x=0,y=0,L={$S(2,3,21)$},style=draw]{dddS}
\Vertex[x=3.5,y=0,L={$S(3,4,22)$},style=draw]{ddS}
\Vertex[x=7,y=0,L={$S(4,5,23)$},style=draw]{dS}
\Vertex[x=10.5,y=0,L={$S(5,6,24)$},style=draw]{S}
\draw (dddS) edge[-{Stealth}] node[font=\tiny,above]{1}  (ddS);
\draw (dddS) edge[":3\kern 7mm 22:"] (ddS);
\draw (ddS) edge[-{Stealth}] node[font=\tiny,above]{11} (dS);
\draw (ddS) edge[":2\kern 7mm 23:"] (dS);
\draw (dS) edge[-{Stealth}] node[font=\tiny,above]{253} (S);
\draw (dS) edge[":1\kern 7mm 24:"] (S);
\end{tikzpicture}
\end{center}
\caption{Derived designs of the three known $S(5,6,24)$ designs.}\label{fig2}
\end{figure}

Denniston~\cite{RHFD76} established the existence of
$S(5,6,24)$ designs by prescribing $\PSL(2,23)$ as
a group of automorphisms. It is known that there are
exactly three non-isomorphic designs under these
assumptions. Figure~\ref{fig2} describes the derived
designs of these three $5$-designs. The order of the
largest possible extension group~$G$ is written
above the arrow, and below are the indices $[\Aut \D:G]$
and $[\Aut \E:G]$. Thus, there are three non-isomorphic
$S(3,4,22)$ designs which we know to be extensible
to $S(4,5,23)$ designs by an extension group of
order~$11$. However, finding these extensions
computationally amounts to exact cover problems
requiring too much CPU time to be solved exhaustively.
Each problem has $630$ items and $1758$ options, out of
which $126$ are to be chosen. Worse still, there
are exactly $\num{26035}$ non-isomorphic $S(3,4,22)$ designs
with an automorphism of order~$11$. If we go from
smaller to larger~$t$, we have no easy way to
distinguish the three extensible designs in the
list of $\num{26035}$ systems $S(3,4,22)$ with a group
of order~$11$. In this case, it is much easier to
construct the $S(4,5,23)$ and $S(5,6,24)$ designs
directly, because larger prescribed groups of
orders $253$ and $6072$ can be used.

To conclude, we believe that we have mostly exhausted
the useful applications of extending Steiner designs
with prescribed groups, as outlined in Sections~\ref{sec3}
and~\ref{sec4}. To make further progress, additional restrictions
derived from new theoretical insights will be needed,
or alternative computational techniques -- such as the
method of tactical decompositions~\cite{KW23, KNP11} --
will have to be implemented and applied.

\end{document}